\documentclass[12pt]{article}

\usepackage[utf8]{inputenc}
\usepackage[T1]{fontenc}
\usepackage{amsmath, amsthm, amssymb}
\usepackage{mathrsfs}
\usepackage{graphicx}
\usepackage{hyperref}
\usepackage{enumitem}
\usepackage{geometry}
\usepackage{authblk}

\newtheorem{theorem}{Theorem}[section]
\newtheorem{lemma}[theorem]{Lemma}
\newtheorem{proposition}[theorem]{Proposition}
\newtheorem{corollary}[theorem]{Corollary}

\theoremstyle{definition}
\newtheorem{definition}[theorem]{Definition}
\newtheorem{remark}[theorem]{Remark}
\newtheorem{example}[theorem]{Example}

\title{\textbf{On Locally Dually and Projectively Flat Almost Rational Finsler Metrics}}
\author[1]{Ghanashyam Kumar Prajapati\thanks{gspbhu@gmail.com}}
\author[2]{Ranadip Gangopadhyay\thanks{gangulyranadip@gmail.com}}
\author[3]{Bankteshwar Tiwari\thanks{banktesht@gmail.com}}

\affil[1]{Department of Applied Science \& Humanities, LNJPIT, Chapra, Bihar, 841302, India}
\affil[2]{Department of Applied Science \& Humanities, VIGNAN'S Foundation for Science, Technology and Research (Deemed to be University), Off Campus Hyderabad, Hyderabad, 501512, Telangana, India}
\affil[3]{DST-CIMS, Institute of Science, Banaras Hindu University, Varanasi, 221005, India}
\date{}

\begin{document}
	
	\maketitle
	
\begin{abstract}
In this paper, we investigate Almost Rational Finsler (AR-Finsler) metrics, a class of Finsler metrics whose fundamental tensor admits a decomposition $g_{ij}(x,y)=\eta(x,y)a_{ij}(x,y),$
where $\eta$ is a positive smooth function and $a_{ij}$ is rational in the fiber variables.  We derive necessary and sufficient conditions for local dual flatness and local projective flatness of AR-Finsler metrics. Furthermore, we obtain a compatibility relation characterizing AR-Finsler metrics that are simultaneously locally dually flat and locally projectively flat. As an application, we establish a rigidity result for AR-Finsler. Motivated by these results and several known rigidity phenomena in special classes of Finsler metrics, we formulate a conjecture concerning the local Minkowskianity of AR-Finsler metrics that are both locally dually flat and locally projectively flat. Several examples are presented to illustrate the theory.
\end{abstract}



\section{Introduction}
The study of locally dually flat and projectively flat Finsler metrics has attracted considerable attention because of its connections with both theoretical and applied mathematics. 
Dually flat structures were introduced in information geometry by Amari and Nagaoka \cite{amari2000}, and later extended to the Finsler setting by Shen \cite{shen2006}. 
Projectively flat Finsler metrics, characterized by the property that their geodesics are straight lines in suitable local coordinates, have been extensively studied since the pioneering work of Hamel \cite{hamel1903} and continue to play a central role in Finsler geometry. Recently, Li et al. \cite{LHZZ} studied dually flat and projectively flat Minkowskian product Finsler metrics. The Funk metric is one of the most prominent non-Riemannian examples that is simultaneously dually and projectively flat.

In the Riemannian setting, projective flatness is equivalent to constant sectional curvature. In Finsler geometry, however, projective flatness gives rise to a rich class of non-Riemannian metrics and strong rigidity phenomena. Understanding the interaction between projective flatness and other geometric properties, such as dual flatness, remains an active area of research.

The class of $m$-th root Finsler metrics provides an important source of non-Riemannian examples. Yu and You \cite{YY10} showed that the components of the fundamental tensor $g_{ij}$ of an $m$-th root metric are generally not rational functions of the fiber coordinates $y$, whereas the spray coefficients ($G^i$) are rational in $y$. Using this rational structure, Tayebi et al. \cite{AB11, AAE14, AAE} investigated various curvature properties of $m$-th root metrics. Subsequently, Tiwari, Prajapati \cite{BG,BG1} studied Kropina changes of $m$-th root metrics and obtained further examples whose spray coefficients remain rational in $y$, although the fundamental tensor itself is not necessarily rational. Motivated from these examples, B. Tiwari et.al.   \cite{BGR} have studied the Finsler metric in the following form:
\begin{equation}\label{eq1.1}
{F}=\sqrt{g_{ij}y^iy^j}
\end{equation} where
\begin{equation}\label{eq1.2}
g_{ij}=\eta(x,y)a_{ij}(x,y),
\end{equation}
$\eta(x,y)$ is an irrational function in $y$ and $a_{ij}(x,y)$ are rational functions in $y$.

This led to the introduction of Almost Rational Finsler metrics (AR-Finsler metrics) by Taha and Tiwari \cite{taha2023}, where rationality is imposed only on certain geometric objects rather than the entire fundamental tensor. Using the rationality of geometric quantities such as the Cartan torsion, geodesic spray, Landsberg curvature, and S-curvature, they proved several rigidity results. In particular, they showed that the S-curvature vanishes for metrics with isotropic S-curvature, that an AR-Finsler metric is weakly Landsberg whenever it has isotropic mean Landsberg curvature, and that an Einstein AR-Finsler metric is Ricci-flat.

The paper is organized as follows. In Section 2, we recall the basic notions. In Section 3, we investigate the local dual flatness of AR-Finsler metrics and we establish a characterization of locally dually flat AR-Finsler metrics and derive several consequences. In Section 4, we study locally projectively flat AR-Finsler metrics and obtain a corresponding characterization in terms of the AR-decomposition. Furthermore, we examine the simultaneous occurrence of local dual flatness and projective flatness and derive a compatibility relation connecting these two geometric structures. Motivated by the absence of known nontrivial examples and the restrictive nature of the obtained compatibility equation, we propose the following conjecture.

\medskip

\noindent\textbf{Conjecture.}
Let $F$ be a non-Riemannian AR-Finsler metric on a smooth manifold $M$. If $F$ is both locally dually flat and locally projectively flat, then $F$ is locally Minkowskian.

\section{Preliminaries}
Let $ M $ be an $n$-dimensional $C^{\infty}$-manifold, $T_{x}M$
denotes the tangent space of $M$ at $x$. The tangent bundle $ TM $
is the union of tangent spaces, $ TM:= \bigcup _{x \in M}T_xM $.
We denote the elements of $TM$ by $(x,y)$, where $x=(x^i)$ is a
point of $M$ and $y \in T_{x}M $.
We denote $TM_0=TM \setminus\left\lbrace 0\right\rbrace $.
\begin{definition}
	\cite{SSZ} \textnormal{A Finsler metric on $M$ is a function $F:TM \to [0,\infty)$ satisfying the following conditions:
		\begin{itemize}
			\item[(i)] $F$ is smooth on $TM_{0}$,
			\item[(ii)] $F$ is a positively 1-homogeneous on the fibers of tangent bundle $TM$,
			\item[(iii)]The Hessian of $\frac{F^2}{2}$ with element $g_{ij}=\frac{1}{2}\frac{\partial ^2F^2}{\partial y^i \partial y^j}$ is positive definite on $TM_0$.
		\end{itemize}
		The pair $(M,F)$ is called a Finsler space and $g_{ij}$ is called the fundamental tensor. }
	\end{definition}
	
	\begin{definition}
		\cite{taha2023} \textnormal{ A Finsler metric $F$ on a manifold $M$  is said to be an Almost Rational Finsler metric if the components $g_{ij}(x,y)$ of its metric tensor are expressed in the form 
			\begin{equation}\label{eq1.2}
			{g}_{ij}(x,y)=\eta(x,y)a_{ij}(x,y),
			\end{equation}
			where,
			\begin{enumerate}
				\item[(i)] $\eta : TM \xrightarrow{} (0,\infty) $ is a smooth function,
				\item[(ii)] the matrix $(a_{ij}(x,y)), 1\leq i,j \leq n$ is symmetric positive definite,
				\item[(iii)] for each $i,j,~~\{\eta(x,y),a_{ij}(x,y)\} $ are  positive homogeneous of degree zero function in the fiber coordinate $y$ and $a_{ij}(x,y)$ is a rational function in $y$.
			\end{enumerate}
			The pair $(M,F)$ is said to be an $AR$-Finsler manifold. Also, if $\eta$ is a rational function in $y$, then the pair $(M,F)$ is called a rational Finsler manifold.}
		\end{definition}
		
		The class of AR-Finsler metrics contains several important families of Finsler metrics.
		
		\begin{example}
			\begin{enumerate}
				
				\item[(i)]\cite{taha2023} \textnormal{The \(m\)-th root metrics
					$F=A^{1/m}$, 
					where $A=a_{i_1\cdots i_m}(x)y^{i_1}\cdots y^{i_m}$
					is a homogeneous polynomial of degree \(m\) in \(y\) is an AR-Finsler metric.}
					
					\item[(ii)]\cite{BG1}\textnormal{The generalized \(m\)-th root metrics $F=\sqrt{A^{2/m}+B},$ where $A=a_{i_1\cdots i_m}(x)y^{i_1}\cdots y^{i_m},$
						$B=b_{ij}(x)y^iy^j$ is an AR-Finsler metric.}
						
						\item[(iii)] \textnormal{Conformal deformations of AR-Finsler metrics is again an AR-Finsler metric. Let \(F\) be an AR-Finsler metric satisfying $g_{ij}=\eta a_{ij},$
							where \(a_{ij}\) is rational in \(y\). Consider the conformal deformation $\bar F=e^{\sigma(x)}F.$
							Its fundamental tensor is $\bar g_{ij}
							=
							e^{2\sigma(x)}g_{ij}
							=
							\bigl(e^{2\sigma(x)}\eta\bigr)a_{ij}.$
							Since multiplication by a smooth function of \(x\) preserves the rationality of \(a_{ij}\), we obtain $\bar g_{ij} =\bar\eta\, a_{ij},~
							\bar\eta=e^{2\sigma(x)}\eta.$
							Hence every conformal deformation of an AR-Finsler metric is again an AR-Finsler metric.}
							\item[(iv)] \textnormal{Every locally Minkowski \(m\)-th root metrics is an AR-Finsler metric.}
						\end{enumerate}
					\end{example}
					
					The following proposition seems not to have been explicitly recorded in the literature. It provides a simple sufficient condition for an m-th root metric to be locally Minkowskian.
					\begin{proposition}
						Let $F=A^{1/m}, ~ m\ge 3,$
						be an \(m\)-th root Finsler metric on a manifold \(M\), where $A=a_{i_1\cdots i_m}(x)y^{i_1}\cdots y^{i_m}$ is a homogeneous polynomial of degree \(m\) in \(y\). Let \(g_{ij}\) denote the fundamental tensor of \(F\). If
						\[
						(g_{ij})_{x^k}y^iy^jy^k=0,
						\]
						then \(F\) is locally Minkowskian.
					\end{proposition}
					
					\begin{proof}
						Since $F^2=g_{ij}y^iy^j$ differentiating both sides with respect to \(x^k\), we obtain
						\[
						(F^2)_{x^k}=(g_{ij})_{x^k}y^iy^j.
						\]
						Multiplying by \(y^k\), it follows that
						\[
						(F^2)_{x^k}y^k=(g_{ij})_{x^k}y^iy^jy^k=0.
						\]
						Hence $(F^2)_{x^k}y^k=0.$ Since $F=A^{1/m}$,
						we have
						\[
						F^2=A^{2/m}.
						\]
						Differentiating with respect to \(x^k\),
						\[
						(F^2)_{x^k}
						=
						\frac{2}{m}A^{\frac{2-m}{m}}A_{x^k}.
						\]
						Therefore, we obtain $\frac{2}{m}A^{\frac{2-m}{m}}A_{x^k}y^k=0.$
						Since \(A>0\) on the slit tangent bundle, it follows that $A_{x^k}y^k=0.$
						
						Now
						\[
						A=a_{i_1\cdots i_m}(x)y^{i_1}\cdots y^{i_m},
						\]
						and therefore
						\[
						A_{x^k}
						=
						\frac{\partial a_{i_1\cdots i_m}}{\partial x^k}
						\,y^{i_1}\cdots y^{i_m}.
						\]
						Multiplying by \(y^k\), we have
						\begin{equation}\label{eq3}
						\frac{\partial a_{i_1\cdots i_m}}{\partial x^k}
						\,y^k y^{i_1}\cdots y^{i_m}=0.    
						\end{equation}

						The left-hand side of \eqref{eq3} is a homogeneous polynomial of degree
						\(m+1\) in the variables \(y^1,\ldots,y^n\). Since it vanishes
						identically for all \(y\), every coefficient of this polynomial must vanish. Hence
						\[
						\frac{\partial a_{i_1\cdots i_m}}{\partial x^k}=0,
						\qquad
						\forall\, i_1,\ldots,i_m,\; k.
						\]
						
						Therefore all coefficients of \(A\) are constant in the base variables \(x\), and thus $A=A(y)$. Consequently, $F=A(y)^{1/m}$
						is independent of the position coordinates \(x\). Hence \(F\) is a Minkowski norm in each tangent space and therefore is locally Minkowskian.
					\end{proof}
					
					\section{Locally dually flat AR-Finsler metrics}\label{sec:dually-flat-AR}
					Locally dually flat Finsler metrics constitute a  distinguished class, originating from information geometry. A Finsler metric is said to be locally dually flat if, in suitable local coordinates, its spray coefficients are derived from a potential function, equivalently characterized by the partial differential equation \begin{equation}\label{eq22}
					[F^2]_{x^k y^l} y^k - 2[F^2]_{x^l} = 0.
					\end{equation}

					\begin{theorem}\label{thm:dually-flat-AR-equivalence}
						Let $F$ be an AR--Finsler metric.
						Then $F$ is locally dually flat if and
						only if the totally symmetric $3$--tensor
						\begin{equation}
						U^{(l)}_{ijk}(x,y) \;:=\; (\eta_{y^l})_{x^k}\,a_{ij} + \eta_{y^l}\,(a_{ij})_{x^k}.
						\end{equation}
						satisfies
						\begin{equation}
						U^{(l)}_{(ijk)}(x,y) = 0
						\qquad\text{for all }(x,y)\in TM\setminus\{0\},
						\end{equation}
						where $U_{(ijk)}$ denotes the full symmetrization in the indices $i,j,k$.
					\end{theorem}
					\begin{proof} 
						\begin{equation}
						F^2(x,y)=
						\eta(x,y)\,a_{ij}(x,y)y^iy^j.   
						\end{equation}
						
						Differentiating with respect to $x^k$, we obtain
						
						\begin{equation}
						[F^2]_{x^k}
						=
						\eta_{x^k}a_{ij}y^iy^j
						+
						\eta(a_{ij})_{x^k}y^iy^j.
						\end{equation}
						
						Now differentiate with respect to $y^l$:
						
						\begin{align}
							[F^2]_{x^ky^l}
							&=
							\frac{\partial}{\partial y^l}
							\Big(
							\eta_{x^k}a_{ij}y^iy^j
							\Big)
							+
							\frac{\partial}{\partial y^l}
							\Big(
							\eta(a_{ij})_{x^k}y^iy^j
							\Big)
							\nonumber\\
							&=
							\eta_{x^ky^l}a_{ij}y^iy^j
							+
							\eta_{x^k}(a_{ij})_{y^l}y^iy^j
							+
							2\eta_{x^k}a_{lj}y^j
							\nonumber\\
							&\qquad
							+
							\eta_{y^l}(a_{ij})_{x^k}y^iy^j
							+
							\eta(a_{ij})_{x^ky^l}y^iy^j
							+
							2\eta(a_{lj})_{x^k}y^j.
						\end{align}
						Multiplying by $y^k$ gives
						\begin{align}
							[F^2]_{x^ky^l}y^k
							&=
							\eta_{x^ky^l}a_{ij}y^iy^jy^k
							+
							\eta_{x^k}(a_{ij})_{y^l}y^iy^jy^k
							+
							2\eta_{x^k}a_{lj}y^jy^k
							\nonumber\\
							&\qquad
							+
							\eta_{y^l}(a_{ij})_{x^k}y^iy^jy^k
							+
							\eta(a_{ij})_{x^ky^l}y^iy^jy^k
							+
							2\eta(a_{lj})_{x^k}y^jy^k.
							\label{full}
						\end{align}
						
						Since $a_{ij}(x,y)$ is assumed to be positively homogeneous of degree
						zero in $y$, Euler's theorem yields
						\begin{align}
							(a_{ij})_{y^l}y^l=0, ~~ (a_{ij})_{x^ky^l}y^l=0.
							\label{Euler}
						\end{align}
						
						Therefore \eqref{full} reduces to
						
						\begin{align}
							[F^2]_{x^ky^l}y^k
							=
							\eta_{x^ky^l}a_{ij}y^iy^jy^k
							+
							2\eta_{x^k}a_{lj}y^jy^k
							+
							\eta_{y^l}(a_{ij})_{x^k}y^iy^jy^k
							+
							2\eta(a_{lj})_{x^k}y^jy^k.
						\end{align}
						Substituting into the \eqref{eq22}, we get,
						\[
						\big[(\eta_{y^l})_{x^k}\,a_{ij} + \eta_{y^l}\,(a_{ij})_{x^k}\big] y^i y^j y^k = 0.
						\]                    or                          \begin{equation}\label{eq:6}
						\big(\eta_{y^l} a_{ij}\big)_{x^k}\,y^i y^j y^k = 0,
						\qquad l\in\{1,\dots,n\}
						\end{equation}
						Now we 
						introduce the \(3\)-tensor (for each fixed \(l\)) as 
						\begin{equation}\label{eq:1}
						U^{(l)}_{ijk}(x,y) \;:=\; \big(\eta_{y^l} a_{ij}\big)_{x^k}\,=(\eta_{y^l})_{x^k}\,a_{ij} + \eta_{y^l}\,(a_{ij})_{x^k}.
						\end{equation}
						Then \eqref{eq:6} can be written  compactly as
						\begin{equation}\label{eq:7}
						U^{(l)}_{ijk}(x,y)\,y^i y^j y^k \;=\; 0.
						\end{equation}

						The left-hand side of \eqref{eq:7} is a homogeneous polynomial of degree \(3\)
						in \(y\). A polynomial in \(y\) that vanishes for all \(y\) has all its
						coefficients zero. Hence \eqref{eq:7} is equivalent to the vanishing of the full
						symmetrization of \(U^{(l)}\) in indices \((i,j,k)\):
						\begin{equation}\label{eq:8}
						U^{(l)}_{(ijk)}(x,y)\;=\;0\qquad\text{for all }(x,y),
						\end{equation}
						where \(U_{(ijk)}\) denotes full symmetrization over \(i,j,k\).
					\end{proof} 
					\begin{remark}
						\textnormal{Expanding \eqref{eq:8} gives the symmetric condition, for all indices \(i,j,k\),
							\[
							(\eta_{y^l})_{x^k}\,a_{ij} + (\eta_{y^l})_{x^i}\,a_{jk} + (\eta_{y^l})_{x^j}\,a_{ki}
							\;+\; \eta_{y^l}\big( (a_{ij})_{x^k} + (a_{jk})_{x^i} + (a_{ki})_{x^j}\big)
							\;=\;0.
							\]  
							This is the explicit tensor PDE system equivalent to \eqref{eq:8}.}
						\end{remark}

						\begin{corollary}
							Let $F$ be an AR-Finsler metric. Suppose that the rational tensor $a_{ij}$ is independent of the base coordinates $x$, i.e., $(a_{ij})_{x^{k}}=0$. Then \(F\) is locally dually flat, if and only if
							\begin{equation}
							\eta(x,y)=\Phi(y)+\Psi(x),   
							\end{equation}
							
							where \(\Phi\) is a function of the fiber coordinates only and \(\Psi\) is a function of the base coordinates only.
						\end{corollary}
						
						\begin{proof}
							By Theorem 3.1, local dual flatness is equivalent to $(\eta_{y^{l}})_{x^{k}}a_{ij}+\eta_{y^{l}}(a_{ij})_{x^{k}}=0$. Since \((a_{ij})_{x^{k}}=0\), we obtain $(\eta_{y^{l}})_{x^{k}}a_{ij}=0$.
							
							Therefore, the condition \(U^{(l)}_{(ijk)}=0\) yields
							\[
							(\eta_{y^{l}})_{x^{k}}a_{ij}
							+
							(\eta_{y^{l}})_{x^{i}}a_{jk}
							+
							(\eta_{y^{l}})_{x^{j}}a_{ki}
							=0.
							\]
							
							Contracting with \(y^{i}y^{j}y^{k}\), we obtain
							\[
							(\eta_{y^{l}})_{x^{k}}a_{ij}y^{i}y^{j}y^{k}
							+
							(\eta_{y^{l}})_{x^{i}}a_{jk}y^{i}y^{j}y^{k}
							+
							(\eta_{y^{l}})_{x^{j}}a_{ki}y^{i}y^{j}y^{k}
							=0.
							\]
							
							Since the three terms are identical after relabelling dummy indices, it follows that
							
							\[
							(\eta_{y^{l}})_{x^{k}}a_{ij}y^{i}y^{j}y^{k}=0.
							\]
							
							Because \(a_{ij}\) is positive definite, $a_{ij}y^{i}y^{j}>0,
							\qquad y\neq 0.$
							
							Hence
							\[
							(\eta_{y^{l}})_{x^{k}}y^{k}=0,
							\qquad \forall\, y.
							\]
							
							Since the left-hand side is linear in \(y\), the above identity holding for all \(y\) implies
							\[
							(\eta_{y^{l}})_{x^{k}}=0,
							\qquad \forall\, k,l.
							\]
							
							Integrating with respect to the variables \(x\) and \(y\), we conclude that \(\eta\) splits as
							\[
							\eta(x,y)=\Phi(y)+\Psi(x),
							\]
							where \(\Phi\) depends only on the fiber coordinates and \(\Psi\) depends only on the base coordinates.
						\end{proof}

						\begin{example}\textbf{(The \(m\)-th root metrics)} \textnormal{ Let \(F=\sqrt[m]{A}\) be an \(m\)-th root Finsler metric,}
							\[
							A=a_{i_1\cdots i_m}(x)y^{i_1}\cdots y^{i_m},\qquad m\ge3,
							\]
							and define
							\[
							A_i=\frac{1}{m}\frac{\partial A}{\partial y^i},\qquad
							A_{ij}=\frac{1}{m(m-1)}\frac{\partial^2 A}{\partial y^i\partial y^j}.
							\]
							\textnormal{The fundamental tensor admits the almost rational representation}
							\[
							g_{ij}=F^{2-m}a_{ij},
							\qquad
							a_{ij}=(m-1)A_{ij}-\frac{m-2}{A}A_iA_j.
							\]
							Equivalently,
							\[
							F^2=\eta\,a_{ij}y^iy^j,
							\qquad
							\eta=F^{m-2}.
							\]
							\textnormal{Here the rational tensor \(a_{ij}\)
								constructed from \(A_{ij}\) and \(A_i\) is a homogeneous polynomial of degree \(m-2\) in \(y\). Now we solve the equation}
								\begin{equation}\label{eq:mroot-main}
								(\eta_{y^l}a_{ij})_{x^k}\,y^i y^j y^k=0,
								\qquad l=1,\dots,n.
								\end{equation} or
								\begin{equation}\label{eq:mroot-expanded}
								\Big[(\eta_{y^l})_{x^k}a_{ij}
								+\eta_{y^l}(a_{ij})_{x^k}\Big]y^i y^j y^k=0.
								\end{equation}
							\end{example}

							\textbf{Remark:}
							The condition shows that for \(m\)-th root metrics the equation
							\((\eta_{y^l}a_{ij})_{x^k}y^iy^jy^k=0\) is extremely rigid:
							any \(y\)-dependence of \(\eta\) forces the metric to be locally Minkowskian.
							This rigidity is consistent with known results on Einstein, weak Einstein, and scalar flag curvature \(m\)-th root metrics.

							\begin{example}
								\textnormal{Let
									\begin{equation}
									\psi(x)=\frac12\sum_{i=1}^{n}(x^i)^2+\frac{\lambda}{4}\left(\sum_{i=1}^{n}(x^i)^2 \right)^2,\qquad \lambda>0,    
									\end{equation}
									and define $a_{ij}(x)=\frac{\partial^2\psi}{\partial x^i\partial x^j}.$ Further, let $\eta(x,y)\equiv 1.$ Then $F^2=a_{ij}(x)y^iy^j$.\\
									Let $r^2=\sum_{k=1}^{n}(x^k)^2.$ Then $\psi(x)=\frac12r^2+\frac{\lambda}{4}r^4.$
									Differentiating with respect to $x^i$, and then with $x^j$ we obtain
									\begin{equation}
									a_{ij}=\psi_{ij}=\delta_{ij}+\lambda r^2\delta_{ij}+2\lambda x_ix_j.    
									\end{equation}
									Hence, $a_{ij}=(1+\lambda r^2)\delta_{ij}+2\lambda x_ix_j.$ Consequently, 
									\begin{equation}
									F^2=a_{ij}y^iy^j=(1+\lambda r^2)|y|^2+2\lambda(x\cdot y)^2.   
									\end{equation}
									Since $1+\lambda r^2>0,$ and $2\lambda(x\cdot y)^2\ge 0,$ it follows that  $(a_{ij})$ is positive definite and hence $F$ is an almost rational Finsler metric. Also simple calculations show that $F$ is locally dually flat and $F$ is not locally Minkowskian.}

								\end{example}
								
								\section{Locally Projectively Flat AR-Finsler Metrics} 
								
								Projectively flat Finsler metrics are characterized by having straight-line geodesics in suitable local coordinates. While in Riemannian geometry projective flatness is equivalent to constant sectional curvature, Finsler geometry admits genuinely non-Riemannian examples such as the $m$-th root Finsler metrics. These rigidity properties motivate the study of projective flatness in the broader class of Almost Rational (AR) Finsler metrics, where rationality is imposed only on selected geometric objects instead of the whole fundamental tensor.\\
								It is known \cite{hamel1903} that a Finsler metric is projectively flat if and only if
								\begin{equation}\label{2.3e}
								F_{x^k y^l} y^k - F_{x^l} = 0.
								\end{equation}
								
								\begin{lemma}\label{lem1}
									If $g_{ij}(x,y)$ is the fundamental tensor of a Finsler metric, then
									\[
									(g_{ij})_{x^k y^l}y^i y^j y^k
									=
									-(g_{lj})_{x^k}y^j y^k.
									\]
								\end{lemma}
								
								\begin{proof}
									Since $g_{ij}$ is $0$-homogeneous in $y$,
									$(g_{ij})_{y^l}y^l=0.$
									Differentiating with respect to $x^k$ gives
									\[
									(g_{ij})_{x^k y^l}y^l=-(g_{ij})_{x^k}.
									\]
									Multiplying by $y^i y^j$ and using the symmetry of $g_{ij}$ together with a relabelling of dummy indices yields
									\[
									(g_{ij})_{x^k y^l}y^i y^j y^k
									=
									-(g_{lj})_{x^k}y^j y^k.
									\]
								\end{proof}
								\begin{theorem}\label{4.2} Let $F$ be  an almost rational Finsler metric, then $F$ is
									projectively flat 
									if it satisfies the following condition:
									
									\begin{equation}\label{eq:PF_AR_reduced}
									-\frac{F_{y^l}}{F}
									(\eta a_{ij})_{x^k}y^iy^jy^k
									+
									(\eta a_{lj})_{x^k}y^jy^k
									-
									(\eta a_{ij})_{x^l}y^iy^j
									=0.
									\end{equation}
									
								\end{theorem}
								
								\begin{proof}
									Since
									\[
									F^2=\eta a_{ij}y^iy^j,
									\]
									differentiation with respect to \(x^l\) yields
									\[
									2FF_{x^l}
									=
									(\eta a_{ij})_{x^l}y^iy^j.
									\]
									Hence
									\begin{equation}
									F_{x^l}
									=
									\frac{1}{2F}
									(\eta a_{ij})_{x^l}y^iy^j.
									\label{eq:Fxl}
									\end{equation}
									
									Differentiating \eqref{eq:Fxl} with respect to \(y^l\), we obtain
									\[
									F_{x^ky^l}
									=
									-\frac{F_{y^l}}{2F^2}
									(\eta a_{ij})_{x^k}y^iy^j
									+\frac{1}{2F}
									\frac{\partial}{\partial y^l}
									\Big(
									(\eta a_{ij})_{x^k}y^iy^j
									\Big).
									\]
									
									Since
									\[
									\frac{\partial}{\partial y^l}
									\Big(
									(\eta a_{ij})_{x^k}y^iy^j
									\Big)
									=
									(\eta a_{ij})_{x^ky^l}y^iy^j
									+
									2(\eta a_{lj})_{x^k}y^j,
									\]
									it follows that
									\[
									F_{x^ky^l}
									=
									\frac{1}{2F}
									\left[
									-\frac{F_{y^l}}{F}
									(\eta a_{ij})_{x^k}y^iy^j
									+
									(\eta a_{ij})_{x^ky^l}y^iy^j
									+
									2(\eta a_{lj})_{x^k}y^j
									\right].
									\]
									
									Multiplying by \(y^k\), we get
									\[
									F_{x^ky^l}y^k
									=
									\frac{1}{2F}
									\left[
									-\frac{F_{y^l}}{F}
									(\eta a_{ij})_{x^k}y^iy^jy^k
									+
									(\eta a_{ij})_{x^ky^l}y^iy^jy^k
									+
									2(\eta a_{lj})_{x^k}y^jy^k
									\right].
									\]
									
									Since \(F\) is projectively flat, \eqref{2.3e}
									holds. Substituting the above expressions and multiplying by \(2F\), we obtain
									\begin{equation}\label{26}
									-\frac{F_{y^l}}{F}
									(\eta a_{ij})_{x^k}y^iy^jy^k
									+
									(\eta a_{ij})_{x^ky^l}y^iy^jy^k
									+
									2(\eta a_{lj})_{x^k}y^jy^k
									-
									(\eta a_{ij})_{x^l}y^iy^j
									=0.   
									\end{equation}
									
									Now from Lemma \ref{lem1}, we have
									$(\eta a_{ij})_{x^ky^l}y^iy^jy^k
									=
									-(\eta a_{lj})_{x^k}y^jy^k.$
									Substituting this identity into \eqref{26} gives
									\[
									-\frac{F_{y^l}}{F}
									(\eta a_{ij})_{x^k}y^iy^jy^k
									+
									(\eta a_{lj})_{x^k}y^jy^k
									-
									(\eta a_{ij})_{x^l}y^iy^j
									=0.
									\]
									
									Hence \eqref{eq:PF_AR_reduced} follows.
								\end{proof}
								
								\begin{theorem}
									Let $F$ be an almost rational Finsler metric on a manifold $M$. Suppose that $F$ is projectively flat and satisfies
									\[
									(g_{ij})_{x^k}y^iy^jy^k=0,
									\qquad
									\forall (x,y)\in TM\setminus\{0\}.
									\]
									Then $(g_{ij})_{x^k}=0,$ and consequently $F$ is locally Minkowskian.
								\end{theorem}
								
								\begin{proof}
									Since $F$ is projectively flat, Theorem \ref{4.2} yields
									\begin{equation}
									-\frac{F_{y^\ell}}{F}(g_{ij})_{x^k}y^iy^jy^k
									+(g_{\ell j})_{x^k}y^jy^k
									-(g_{ij})_{x^\ell}y^iy^j
									=0.
									\label{eq:pf}
									\end{equation}
									
									By hypothesis, $(g_{ij})_{x^k}y^iy^jy^k=0,$
									and therefore the first term of \eqref{eq:pf} vanishes identically. Hence
									\begin{equation}
									(g_{\ell j})_{x^k}y^jy^k
									-(g_{ij})_{x^\ell}y^iy^j
									=0.
									\label{eq:qform}
									\end{equation}
									From (29), we have
									\[
									\Big((g_{\ell j})_{x^k}-(g_{jk})_{x^\ell}\Big)y^\ell y^j y^k=0
									\]
									for all \(y\). Since \(y^\ell y^j y^k\) is completely symmetric, only the symmetric part of the coefficient contributes. Hence
									\[
									\Big((g_{\ell j})_{x^k}-(g_{jk})_{x^\ell}\Big)_{(\ell jk)}=0.
									\]
									Expanding the symmetrization and using the symmetry \(g_{ij}=g_{ji}\), we obtain
									\[
									\frac13\Big((g_{\ell j})_{x^k}
									+(g_{\ell k})_{x^j}
									-2(g_{jk})_{x^\ell}\Big)=0.
									\]
									Therefore,
									\[
									(g_{\ell j})_{x^k}
									+(g_{\ell k})_{x^j}
									=
									2(g_{jk})_{x^\ell},
									\]
									which is precisely (30).
									
									Since \eqref{eq:qform}
									holds for all \(y\), the symmetric part of the coefficient tensor must vanish. Therefore,
									\[
									\bigl((g_{\ell j})_{x^k}-(g_{jk})_{x^\ell}\bigr)
									+
									\bigl((g_{\ell k})_{x^j}-(g_{kj})_{x^\ell}\bigr)=0.
									\]

									Using the symmetry \(g_{jk}=g_{kj}\) and above result we obtain,
									\begin{equation}
									(g_{\ell j})_{x^k}
									+
									(g_{\ell k})_{x^j}
									=
									2(g_{jk})_{x^\ell}.
									\label{eq:A}
									\end{equation}
									
									Next, define
									\[
									T_{ijk}:=(g_{ij})_{x^k}.
									\]
									Since $g_{ij}=g_{ji}$, the tensor $T_{ijk}$ is symmetric in the first two indices. The condition
									\[
									T_{ijk}y^iy^jy^k=0
									\]
									for all $y$ implies that the complete symmetrization of $T_{ijk}$ vanishes. Thus
									\begin{equation}
									(g_{ij})_{x^k}
									+
									(g_{jk})_{x^i}
									+
									(g_{ki})_{x^j}
									=
									0.
									\label{eq:B}
									\end{equation}
									
									Now writing \eqref{eq:A} and its cyclic permutations, we obtain
									\begin{align}
										(g_{ij})_{x^k}
										+
										(g_{ik})_{x^j}
										&=
										2(g_{jk})_{x^i},
										\label{eq:C1}
										\\
										(g_{jk})_{x^i}
										+
										(g_{ji})_{x^k}
										&=
										2(g_{ik})_{x^j},
										\label{eq:C2}
										\\
										(g_{ki})_{x^j}
										+
										(g_{kj})_{x^i}
										&=
										2(g_{ij})_{x^k}.
										\label{eq:C3}
									\end{align}
									
									Let
									\[
									A=(g_{ij})_{x^k},
									\qquad
									B=(g_{ik})_{x^j},
									\qquad
									C=(g_{jk})_{x^i}.
									\]
									Then \eqref{eq:C1}--\eqref{eq:C3} become
									\[
									A+B=2C,
									\qquad
									C+A=2B,
									\qquad
									B+C=2A.
									\]
									Which yields,
									\[
									A=B=C,
									\]
									that is,
									\begin{equation}
									(g_{ij})_{x^k}
									=
									(g_{ik})_{x^j}
									=
									(g_{jk})_{x^i}.
									\label{eq:D}
									\end{equation}
									
									Substituting \eqref{eq:D} into \eqref{eq:B}, we obtain
									
									\[
									(g_{ij})_{x^k}=0.
									\]
									
									Finally, differentiating
									\[
									F^2=g_{ij}y^iy^j
									\]
									with respect to $x^k$, we get
									\[
									2FF_{x^k}
									=
									(g_{ij})_{x^k}y^iy^j.
									\]
									Since $(g_{ij})_{x^k}=0$, it follows that $F_{x^k}=0.$
									
									Therefore $F$ is independent of the base coordinates in a neighborhood, and consequently $F$ is locally Minkowskian.
								\end{proof}
								\begin{example}
									\textnormal{
										Let \(F=A^{1/m}\) (\(m\ge3\)) be an \(m\)-th root Finsler metric.
										Suppose that \(F\) is locally projectively flat and satisfies $(\eta a_{ij})_{x^k}y^iy^jy^k=0,$ Then \(F\) is locally Minkowskian.}    
									\end{example}
									
									\begin{example}
										\textnormal{Let \(\mathbb{B}^{n}=\{x\in\mathbb{R}^{n}:|x|<1\}\) be the unit ball and consider the Finsler metric
											\begin{equation*}
												F(x,y)
												=
												\sqrt{
													\frac{(1-|x|^{2})|y|^{2}+(x\cdot y)^{2}}
													{(1-|x|^{2})^{2}}}    
												\end{equation*}
												Then \(F\) is an AR-Finsler metric which is projectively flat but neither locally dually flat nor locally Minkowskian.}
											\end{example}
											
											Writing $F^{2}=g_{ij}(x)y^{i}y^{j},$
											we obtain
											\begin{equation}
											g_{ij}(x)
											=
											\frac{(1-|x|^{2})\delta_{ij}+x_{i}x_{j}}
											{(1-|x|^{2})^{2}}.    
											\end{equation}
											
											Since \(g_{ij}\) is independent of \(y\), we may write
											\[
											g_{ij}(x,y)=\eta(x,y)a_{ij}(x,y),
											\qquad
											\eta\equiv 1,
											\]
											with
											\[
											a_{ij}(x,y)
											=
											\frac{(1-|x|^{2})\delta_{ij}+x_{i}x_{j}}
											{(1-|x|^{2})^{2}}.
											\]
											Hence \(a_{ij}\) is rational in the fiber coordinates \(y\) (indeed independent of \(y\)). Therefore \(F\) is an AR-Finsler metric.
											
											\begin{enumerate}
												\item[1.] \textbf{(\(F\) is not locally dually flat.)} Let
												\[
												A=(1-|x|^{2})|y|^{2}+(x\cdot y)^{2}.
												\]
												Then
												\[
												F^{2}=\frac{A}{(1-|x|^{2})^{2}}.
												\]
												A direct computation yields
												\[
												(F^{2})_{x^{k}y^{l}}y^{k}-2(F^{2})_{x^{l}}
												=
												\frac{-2|y|^{2}x_{l}
													+8(x\cdot y)y_{l}}
													{(1-|x|^{2})^{2}}.
													\]
													Choose
													\[
													x=\left(\frac12,0,\ldots,0\right),
													\qquad
													y=(0,1,0,\ldots,0).
													\]
													Then
													\[
													x\cdot y=0,
													\qquad
													|y|^{2}=1,
													\]
													and therefore
													\[
													(F^{2})_{x^{k}y^{l}}y^{k}-2(F^{2})_{x^{1}}
													=
													-\frac{1}{(1-|x|^{2})^{2}}
													\neq 0.
													\]
													Hence the dual-flatness equation fails and \(F\) is not locally dually flat.
													
													\item[2.] \textbf{(\(F\) is locally projectively flat.)} The metric tensor
													\begin{equation}\label{eq1}
													g_{ij}
													=
													\frac{(1-|x|^{2})\delta_{ij}+x_{i}x_{j}}
													{(1-|x|^{2})^{2}}   
													\end{equation}
													is precisely the metric tensor of the Beltrami--Klein model of the hyperbolic space. It is well known that the geodesics of the Beltrami--Klein metric are Euclidean straight line segments. Therefore, in the coordinates $(x^{1},\ldots,x^{n}),$
													all geodesics are straight lines. By definition, \(F\) is locally projectively flat.
													
													\item[3.] \textbf{(\(F\) is not locally Minkowskian.)} A locally Minkowskian metric is independent of the base coordinates \(x\). However, $g_{ij}(x)$ given by \eqref{eq1} depends explicitly on \(x\), and hence \(F\) is not locally Minkowskian.
													
												\end{enumerate}

												\begin{theorem}
													Let $F$ be an almost rational Finsler metric on a smooth manifold $M$. Suppose that $F$ is both locally dually flat and projectively flat then, for every $l=1,\ldots,n$, $F$ satisfies the following equation:
													\begin{equation}\label{ratio}
													\frac{F_{y^l}}{F}
													=\frac{(g_{ij})_{x^l}y^iy^j}
													{(g_{ij})_{x^k}y^iy^jy^k},
													\end{equation}
												\end{theorem}
												\begin{proof}
													Let $g_{ij}:=\eta a_{ij}.$ Since $F$ is projectively flat, Theorem 4.2 yields
													
													\begin{equation}\label{pf}
													-\frac{F_{y^l}}{F}
													(g_{ij})_{x^k}y^iy^jy^k
													+
													(g_{lj})_{x^k}y^jy^k
													-(g_{ij})_{x^l}y^iy^j= 0.
													\end{equation}
													
													Since $F$ is locally dually flat, Theorem 3.1 implies 
													
													\begin{equation}
													(F^2)_{x^ky^l}y^k
													=2(F^2)_{x^l}.    
													\end{equation}

													Using $F^2=g_{ij}y^iy^j,$ together with Lemma \ref{lem1}, one obtains
													
													\begin{equation}\label{dual}
													(g_{lj})_{x^k}y^jy^k
													=2(g_{ij})_{x^l}y^iy^j.
													\end{equation}
													
													Substituting \eqref{dual} into \eqref{pf} gives $-\frac{F_{y^l}}{F}
													(g_{ij})_{x^k}y^iy^jy^k
													+
													(g_{ij})_{x^l}y^iy^j
													=0.$
													Hence, the theorem follows, whenever the denominator is nonzero.
												\end{proof}
												
												\begin{corollary}
													Let $F$ be an almost rational Finsler metric on a smooth manifold $M$. If $F$ is both projectively flat and locally dually flat and satisfies $(\eta a_{ij})_{x^k}y^iy^jy^k=0,$ then $F$ is locally Minkowskian.
												\end{corollary}
												\begin{proposition}
													Let $F$ be a locally dually flat and projectively flat AR Finsler metric on a smooth manifold $M$, ($\dim~M\ge 2$) such that $\eta(x,y)=\eta(x)$ and $a_{ij}(x,y)=a_{ij}(y)$ then $F$ cannot be a regular strongly convex Finsler metric.
												\end{proposition}
												
												\begin{proof}
Assume $\eta_{x^k}y^k\neq 0$ on the domain under consideration.	Since $a_{ij}$ depends only on $y$, $(a_{ij})_{x^l}=0, ~~ \forall l$, from \eqref{ratio} we obtain 
\begin{align*}
	\frac{F_{y^l}}{F} &= \frac{(g_{ij})_{x^l} y^i y^j}{(g_{ij})_{x^k} y^i y^j y^k} = \frac{(\eta a_{ij})_{x^l} y^i y^j}{(\eta a_{ij})_{x^k} y^i y^j y^k} = \frac{(\eta_{x^l}) a_{ij} y^i y^j + \eta (a_{ij})_{x^l} y^i y^j}{(\eta_{x^k}) a_{ij} y^i y^j y^k + \eta (a_{ij})_{x^k} y^i y^j y^k} \\[2ex]
	\intertext{Since $a_{ij}(x,y) = a_{ij}(y) \implies (a_{ij})_{x^l} = 0$ and $(a_{ij})_{x^k} = 0$:}
	\therefore \frac{F_{y^l}}{F} &= \frac{(\eta)_{x^l} a_{ij} y^i y^j}{(\eta)_{x^k} a_{ij} y^i y^j y^k} = \frac{(\eta_{x^l}) (a_{ij} y^i y^j)}{(\eta)_{x^k} (a_{ij} y^i y^j) y^k}= \frac{(\eta_{x^l})}{(\eta)_{x^k} y^k}
\end{align*}
													The above equation can be rewritten as
													\begin{equation}
													(\ln F)_{y^l}
													=
													\frac{\eta_{x^l}}
													{\eta_{x^k}y^k}.
													\end{equation}
													
													Since $\frac{\partial}{\partial y^l}
													\ln(\eta_{x^k}y^k)= \frac{\eta_{x^l}}
													{\eta_{x^k}y^k},$ it follows that
													
													\begin{equation}
													F(x,y)
													=
													\phi(x)\,\eta_{x^k}(x)y^k,
													\end{equation}
													
													Now
													\begin{equation}
													F^2
													=
													\phi(x)^2\bigl(\eta_{x^k}y^k\bigr)^2.
													\end{equation}
													
													Differentiating twice with respect to $y^i$ and $y^j$, we obtain
													\begin{equation}
													g_{ij}
													=
													\frac12(F^2)_{y^iy^j}
													=
													\phi(x)^2\eta_{x^i}\eta_{x^j}.
													\end{equation}
													
													Hence $g_{ij}$ is the outer product of the vector
													$(\eta_{x^1},\ldots,\eta_{x^n})$ with itself. Therefore
													\begin{equation}
													\operatorname{rank}(g_{ij})
													\le 1.
													\end{equation}
													
													Since $n\ge 2$, we have
													\begin{equation}
													\det(g_{ij})=0.
													\end{equation}
													
													Thus $g_{ij}$ is degenerate and cannot be positive definite.
													Consequently, $F$ is not a regular strongly convex Finsler metric.
												\end{proof}




\begin{thebibliography}{00}
													\bibitem{amari2000} S.-I. Amari, H. Nagaoka, Methods of Information Geometry, AMS Monographs, Oxford University Press, 2000.
													\bibitem{shen2006} Z. Shen, Riemann-Finsler geometry with applications to information geometry, Chin. Ann. Math., 27B(1) (2006), 73--94.
													\bibitem{hamel1903} G. Hamel, Über die Geometrien, in denen die Geraden die Kürzesten sind, Math. Ann., 57 (1903), 231--264.
													\bibitem{taha2023} E.H. Taha, B. Tiwari, On Almost Rational Finsler metrics, Bull. Iranian Math. Soc., 49(8) (2023), 1--15.
													\bibitem{cheng2011} X. Cheng, Y. Tian, Locally dually flat Finsler metrics with special curvature properties, Differential Geom. Appl., 29
													(2011), 98--106.
													
													\bibitem{YY10}  Y. Yu and Y. You, On Einstein $m$-th root metrics , Differential Geom. Appl., 28 (2010), 290--294.
													\bibitem{AB11} A. Tayebi and B. Najafi, On $m$-th root metrics with special curvature properties, C. R. Acad. Sci. Paris, Ser. I, 349 (2011), 691--693.
													\bibitem{AAE14} A. Tayebi, A. Nankali and E. Peyghan, Some Properties of $m$-th root Finsler Metrics, J. Contemp. Math. Anal., 49 (2014), 157--166.
													\bibitem{AAE} A. Tayebi, A. Nankali and E. Peyghan, Some curvature properties of Cartan spaces with $m$-th root metrics, Lithuanian Math. J., 54 (2014), 106--114.
													\bibitem{BG} B. Tiwari and G. K. Prajapati, On Einstein Kropina change of $m$-th root Finsler metrics, Differ. Geom. Dyn. Syst., 18 (2016), 139--146.
													\bibitem{BG1} B. Tiwari and G. K. Prajapati, On generalized Kropina change of $m$-th root Finsler metrics, Int. J. Geom. Methods Mod. Phys., 14 (2017) 1750081 (11 pages).
													\bibitem{BGR} B. Tiwari, G. K. Prajapati and R. Gangopadhyay, On Finsler spaces with rational spray coefficients, Differ. Geom. Dyn. Syst., 21 (2019), 180--188.
													\bibitem{SSZ} Chern, S.S., Shen, Z.: Riemannian-Finsler geometry. World Scientific Publisher, Singapore (2005).
													\bibitem{LHZZ} Li, J., He, Y., Zhang, X., Zhang, N.: Dually flat and projectively flat Minkowskian product Finsler metrics, J. Math. Anal. Appl., 524(2) (2023).
												\end{thebibliography}
											\end{document}